\documentclass[12pt]{amsart}
\usepackage{amscd}
\usepackage[all]{xy}
\usepackage[notcite,notref]{showkeys}

 \newcommand{\Br}{\operatorname{Br}}

\newcommand{\rank}{{\operatorname{rank}}}

\newcommand{\CH}{\operatorname{CH}}

\newcommand{\Sym}{\operatorname{Sym}}

\newcommand{\Bl}{\operatorname{Bl}}

\newcommand{\Var}{\operatorname{Var}}

\newcommand{\C}{\mathbf{C}}
\renewcommand{\L}{\mathbf{L}}
\renewcommand{\P}{\mathbf{P}}
\newcommand{\N}{\mathbf{N}}
\newcommand{\Q}{\mathbf{Q}}
\newcommand{\Z}{\mathbf{Z}}
\newcommand{\un}{\mathbf{1}}
\newcommand{\A}{\mathbf{A}}

 \newcommand{\sC}{\mathcal{C}}
\newcommand{\sG}{\mathcal{G}}
\newcommand{\sM}{\mathcal{M}}

\newcommand{\sO}{\mathcal{O}}

 \newcommand{\sF}{\mathcal{F}}
\newcommand{\sA}{\mathcal{A}}

 \newcommand{\sS}{\mathcal{S}}

 \newcommand{\sT}{\mathcal{T}}

 \numberwithin{equation}{section}

\theoremstyle{plain}
\newtheorem{thm}[equation]{Theorem}
\newtheorem{prop}[equation]{Proposition}

\newtheorem{conj}[equation]{Conjecture}

\theoremstyle{definition}

\newtheorem{ex}[equation]{Example}
\newtheorem{rk}[equation]{Remark}

\begin{document}
 \title{Kuznetsov components and transcendental motives of cubic fourfolds}
 \author[C. Pedrini]{Claudio Pedrini}
\address{Dipartimento di Matematica \\ %
Universit\'a degli Studi di Genova \\ %
Via Dodecaneso 35 \\ %
16146 Genova \\ %
Italy}
\email{claudiopedrini4@gmail.com}

\begin{abstract} Let $X \subset \P^5_{\C}$ be a smooth cubic fourfold.The Kuznetsov component $\sA_X$ is contained in the derived category $D^b(X)$ and the transcendental  motive $t(X)$  is contained in  the category of Chow
motives $\sM_{rat}(\C))$. If $X$ and $Y$ are {\it Fourier -Mukai partners}  and hence the categories $\sA_X$ and $\sA_Y$ are equivalent,  then their transcendental motives  $t(X)$ and $t(Y)$ are isomorphic.
The aim of this note is to consider  families of special  cubic fourfolds $X$  with their FM-partners  $Y$ and to give an explicit description of the  isomorphism between the transcendental motives, in the case $X$ and $Y$ are rational and when they are conjecturally irrational. We also prove that ,for special cubic fourfolds $X $ in countably many Hassett divisors, with a symplectic automorphism of order 3,  there exists another special  cubic fourfold $Y$, an equivalence of categories $\sA^G_X \simeq \sA_{Y}$, where $\sA^G_X$ is the equivariant Kuznetsov component, and an isomorphism $t(X) \simeq t(Y)$. 

\end{abstract}
\maketitle 

\section{Introduction}

Let $X$ be a complex smooth cubic fourfold in $\P^5$. Let $A(X)=H^4(X,\Z)\cap H^{2,2}(X)$ be the lattice of algebraic cycles. The fourfold $X$ is {\it special} if  the lattice $A(X)$ contains a class  $v$ that is not homologous to $h^2$, where $h$ is the class of a hyperplane section. A fourfold $X$ is  {\it very general} if $\rank A(X) =1$. For $X$ special let $ K_d=<h^2,v>$, where $d$ is the discriminant of $<h^2,v>$.\par
 A polarized K3 surface  $(S,L)$ of  degree $d$ and genus $g$,  with $L^2= 2g-2$ and $2g-2 =d$, is  said  to be associated to a cubic fourfold $X \in \sC_d$ if there is an isomorphism of Hodge structures
\begin{equation} \label{K3} K^{\perp}_d \simeq H^2(S, \Z)_{prim}(-1). \end{equation}
Let $\sC$ be the moduli space of smooth projective cubic fourfolds and let  $\sC_d \subset \sC$ be the {\it Hassett divisor }  parametrizing  special  cubic fourfolds  with a labelling of discriminant d, i.e. a positive defined rank-two primitive sublattice $K_d=<h^2,v> \subset A(X)$.
 B. Hassett  in [Hass] proved  that $X \in \sC_d$ has an associated  polarized K3 surface  of degree $d$ if and only if $d$ satisfies the following numerical condition\par
(*) $d>6$ and  $d$ is not divisible by 4,9 or a prime $p \equiv 2 (3)$.\par
The results of several Authors (Hassett,Kuznetsov, Addington-Thomas) suggest
that a cubic fourfold is rational if and only if it has an associated
K3 surface. Kuznetsov conjectured that X has an associated K3 surface
if and only if there exists a semi-orthogonal decomposition of the derived category $D^b(X)$ of bounded complexes of coherent sheaves
\begin{equation}\label{Kuz} D^b(X)=< \sA_X, \sO_X(1),\sO_X(2)>\end{equation}
such that  the {\it  Kuznetsov  component } $\sA_X,$ is equivalent to the category $D^b(S)$,where $ S$ is K3 surface.The following result proved in [BLMNPS, Cor.29.7], shows that the two conditions are in fact equivalent
\begin{thm}\label{FMequiv} Let $X$ be cubic fourfold. Then $X$ has a Hodge-theoretically associated K3 surface if and only if there exists a smooth projective K3  surface $S$ and an equivalence $\sA_X \simeq D^b(S)$\end{thm}
If $d$ satisfies (*) the set of cubics $X \in \sC_d$ such that $\sA_X \simeq D^b(S)$ is Zariski open dense, see[AT,Thm.1.1].\par
These results leaded to the following Conjecture
\begin{conj} \label{rat}  A smooth cubic fourfold $X$ is rational if and only if $X \in \sC_d$  with $d$ satisfying the  numerical condition in (*) \par
\end{conj}
A recent result in  [KKPY,Thm. 6.8] proves that a very general cubic fourfold $X$,in the sense that the only rational Hodge classes on $X$ are the powers of the hyperplane class, is not rational.\par
The results by Hassett and Kuznetsov have been extended by  Huybrechts in [Huy 1],  showing that if  $X\in \sC_d$ with 
$$ d\equiv 0,2 (6)   \   and  \  n_i \equiv 0 (2)  \  for  \  all  \ primes \  p_i \equiv 2 (3)  \   with   \  2d=\prod p_i^{n_i}$$
\noindent then   there is a equivalence of categories
\begin{equation} \label{Brauer} \sA_X \simeq D^b(S,\alpha),\end{equation}
 \noindent with $S$ a K3 surface and $\alpha$ a Brauer class in $\Br(S)$. For $X \in \sC_8$ the K3 surface $S$ has degree 2 and $\alpha $ has order 2 while for $X \in \sC_{18}$  the degree of $S$ is 2 and the order of   $\alpha$  is 3.
A recent result by B.Hassett  describes a divisor in $\sC_{24}$ consisting of cubic fourfolds $X$ with twisted K3 surfaces $(S,\alpha)$ where $S$ has degree 6 and $\alpha $ has order 2. In all these cases the fourfolds are rational whenever $\alpha =0$.\par
 Let $\sM_{rat}(\C)$ be the (covariant) category of Chow motives. The motive $h(X)$ of a cubic fourfold $X$ has a {\it reduced Chow K\"unneth decomposition }
\begin{equation}\label{decomp} h(X) = \un \oplus \L \oplus (\L^2)^{\oplus \rho(X)} \oplus t(X) \oplus \L^3 \oplus \L^4,\end{equation}
\noindent where  $\L$ is the Lefschetz motive, $\rho(X) =\rank A^2(X)$, with $A^2(X)= CH^2(X)\otimes \Q$ and $t(X)$ is  the transcendental motive, see [BP]. Then
$$H^*(t(X)) = H^4_{tr}(X,\Q) =T(X)_{\Q},$$
where $T(X)\subset H^4(X.\Z)$   is the transcendental lattice and
$$A^i(t(X)) =0   \  for   \  i\ne 3  \  ;   \  A^3(t(X)) =A_1(X)_{hom} =A_1(X)_{alg},$$
Similarly if $S$ is a K3 surface then its motive $h(S)$ has a reduced Chow-K\"unneth  decomposition as follows
$$h(S)=\un \oplus \L^{\oplus \rho(S)}\oplus t_2(S) \oplus \L^2$$
\noindent where $\rho(S)$ is the rank of the Neron-Severi $NS(S)$, see [KMP].
  B\"ulles in  [Bull,3.1] proved that, if $X \in \sC_d$ with $d$ satisfying the following numerical condition 
$$(**)  \exists f,g \in \Z   \   with \   g\vert 2n^2+2n+2 \  ,  \ n \in \N  \  and   \ d =f^2g $$
\noindent  there is an equivalence of  motives
\begin{equation}\label{K3iso} t(X) \simeq t_2(S)(1),\end{equation} 
\noindent where $S$ is a K3 surface.\par
Two cubic fourfolds $X$ and $Y$ are {\it Fourier-Mukai partners} ( in short FM partners) if there exists an equivalence $\sA_X \simeq \sA_Y$  such th
$$D^b(X) \to \sA_X \simeq \sA_Y \to D^b(Y)$$
\noindent is a Fourier-Mukai transform.\par
L.Fu and Ch.Vial in [FV,Thm.1]proved that two smooth cubic fourfolds $X,Y$ with Fourier-Mukai equivalent  Kuznetsov components $\sA_X \simeq\sA_Y$ have isomorphic Chow motives. Therefore
\begin{equation} \label{Kuz} \sA_X \simeq \sA_Y  \Rightarrow t(X) \simeq t(Y) \end{equation}
The number of FM -partners of a cubic fourfold $X$, up to isomorphism, is finite and equals to 1 if  $X$ is not contained in any Hassett divisor $\sC_d$, see [Huy 2, 7.3.18].\par 
 D.Huybrechts in [Huy 2 ,Chapter 7] conjectured that if  two cubic fourfolds  are FM-partners then they are birational.  In all the known cases of cubic fourfolds $X,Y$, such that $\sA_X \simeq \sA_Y$, the fourfolds are in fact birational, see
  [BGvBM]. A counterexample to the converse is given by two general Pfaffian cubic fourfolds, i.e. two general cubic fourfolds in the
divisor $\sC_{14} $. In fact, any two Pfaffian cubic fourfolds are birational since they are both rational by but they are in general not FM partners, see [BLiB,Rk.3.7].\par
 Note that, until the results on the irrationality of cubic fourfolds proved in [KKPY] no pair of cubic  fourfolds were known to be birationally inequivalent. In [KKPY, Ex. 6.17] it is proved that a very general $X \in \sC_8$ is irrational. Therefore
$X$ is not  birational equivalent to a  fourfold $Y \in \sC_{14} $, which is known to be rational.\par
 The aim  of this note is to describe the isomorphism between the transcendental motives of  two FM partners  $X,Y$ belonging to a Hassett divisor $\sC_d$.\par
 In Sect.2 we consider  FM-partners $X,Y$, where $X,Y$ belong  either to the Hassett divisor $\sC_{12}$ or to $\sC_{42}$. In the first case the cubics are conjecturally irrational while in the second case they are both rational. We
 show that the the equivalence $\sA_X \simeq \sA_Y$ yields an isomorphism $t(X) \simeq $t(Y) which is in any case induced by an isomorphism between the transcendental motives of  K3 surfaces associated to $X$ and $Y$, as in
\ref{K3iso}. We also consider the case of cubic fourfolds $X,X'$ in $\sC_{12}$ such that the Fano varieties $F(X)$ and $F(X')$ are birational and describe the isomorphism btween $t(X)$ and $t(X')$.\par
  In Sect 3 we  describe the action of a finite group  $G$ of automorphisms on the Kuznetsov  component  $\sA_X$ and  give examples where $\sA^G_X$ is equivalent to the derived category of either a K3 surface or an abelian surface.
 We also prove that, for $X$ in countably  many Hasset divisors,  there is a cubic fourfold $Y$ such that
$$\sA^G_X \simeq \sA_Y   \  and   \ t(X) \simeq t(Y),$$
\noindent where $\sA^G_X $ is the equivariant Kuznetsov  component.
\section{Fourier-Mukai partners}
Let  $X$  be   a  very general   cubic fourfold  in a Hassett divisor  $ \sC_d$, i.e. such  that the lattice of algebraic cycles has rank 2.  If $d$ is not divisible to 9 then also a FM-partner $Y$ belongs to $\sC_d$ and is general  
since $T(X) \simeq T(Y)$, where  $T(X)$ and $T(Y)$ are the lattices of transcendental cycles,, see [FL,2.3].\par
 In the case of a  very general   cubic fourfold  $X \in \sC_d$  the number of FM-partners of  $X$ can be deduced from the following result (see [FL, Prop.2.6])
\begin{prop}\label{FM} Let $X$  be a very general cubic fourfold in $\sC_d$, where $d \ge 8$, $d \equiv 0,2(6)$ and is not divisible by 9. Let \par
$m =1$ if $d =2^a$;\par
$m=2^{k-1}$ if $d =2p_1\cdot....\cdot p_k$;\par
$m= 2^k$ if  $d =2^a p_1\cdot....\cdot p_k$.\par 
Here  $a \ge 2$ and $p_1,\cdots,p_k$ are distinct odd primes.Then \par
$ (1) If d\equiv 2(6) $, then  the number of Fourier-Mukai partners of $X $ equals $m$;\par
(2)  if $d \equiv 0(6)$  then  the number of Fourier-Mukai partners of $X $ equals $m/2$.
\end{prop}
\begin{ex}  If $X$ is very general in $\sC_{20} $ then $X$  contains a Veronese surface $V$ . The cubic $X$ has   one  FM partner $X' \in \sC_{20}$ not isomorphic to $X$ , containing a Veronese surface $V'$ and birational  to $X$
In this case, since $d=20$ does not satisfy the numerical condition in (**), the cubic fourfolds $X,X'$ have no associated K3 surfaces, in the motivic sense.The cubic $X'$ is  obtained as the image of $X$ under the Cremona transformation  $F_V$ of $\P^5_{\C} $ defined by the system of quadrics passing trough the Veronese surface $V$. In [FL]  it is proved that $\sA_X \simeq \sA_{X'}$. The restriction of $F_V$ to $X$ produces a birational map
$$f_V :X \dashrightarrow X'$$
 In this case the automorphism  $t(X) \simeq t(X')$ can be described as follows. Let 
$\pi : Y \to X$ and $\pi' : Y \to X'$  be the blow-ups at  $V$ and $V'$, respectively, see [FL,(3.2)]. By Manin's formula we get isomorphisms in $\sM_{rat}(\C)$
\begin{equation} \label{blow-up} h(Y) \simeq h(X) \oplus h(V)(1) \simeq h(X') \oplus h(V') (1)\end {equation}
The surfaces $V,V'$ being rational, their motives $h(V)$ and $h(V')$ have no transcendental parts. Therefore the isomorphism  in \ref{blow-up} induces, via a Chow-K\"unneth decompositions of $h(X)$ and $h(X')$ (see \ref{decomp}),
an isomorphism $t(X) \simeq t(X')$. In this case $X$ and $X'$ are conjecturally irrational.A similar result also holds for rational cubics in  $X,X' \in \sC_{20}$.Taking $X \in \sC_{20} \cap \sC_{38}$  then $X$ is rational and the Cremona transformation produces a FM-partner $X'$ lying in a codimension 2 subvariety of $\sC_{20} \cap \sC_{62}$,see [FL,Prop.3.15]. The fourfold $X'$, being birational to $X$ is rational.
\end{ex}
According to \ref{FM} a very general fourfold $X$ belonging either to  $\sC_{12}$ or to  $\sC_{42}$ has no FM-partners $Y$, with $Y \ne X$ . From  the results in  [BGvBM] there are families of cubic fourfolds  $X$, admitting an automorphism of order 2 or  3, belonging  either  to $\sC_{12}$  or  to $\sC_{42}$, with a large number of Fourier-Mukai partners $Y$. If $X,Y \in \sC_{12} $ they are conjecturally irrational while $X,Y$ are rational if they  belong to $\sC_{42}$.\par
Here we show that  in all these  cases  the   FM partners  $X,Y$ have associated K3 surfaces  $S, S'$ ( in the motivic sense)  and  
$$t(X) \simeq t_2(S)(1)  \   ;   \    t(Y) \simeq t_2(S')(1) \  with  \ t_2(S) \simeq t_2(S') $$
Therefore the isomorphism between the transcendental motives of $X$ and $Y$ is induced by an isomorphism of the transcendental motives of the associated K3 surfaces.\par
 Let $\sigma$ be  the involution on $\P^5$:
$$\sigma :  [x_0,x_1,x_2,x_3,x_4,x_5]  \to  [x_0,x_1,x_2,x_3,- x_4,- x_5] $$
\noindent and let   $\sF $ be the family of cubic fourfolds  invariant under the  symplectic involution induced by $\sigma$ . The family $\sF$  has dimension 12.  If $X \in \sF$ the  lattice   $A(X)_{prim} \simeq  E_8(2)$  has rank 8 . A fourfold $X \in \sF$  contains 120 pairs  of cubic scrolls $\{T_i,T^*_i\}$. The lattice $M$ generated by $\alpha_i = [T_i] -h^2$, where $h$ is the class of a hyperplane section, is isomorphic to $A(X)_{prim}$, see [Marq,4.9].  Therefore $X \in \sC_{12}$. The fourfold $X$ contains no planes and has no associated K3 surface, in the Hodge theoretical sense.\par
The fixed locus of $\sigma$ on $X$ consists of a line $l$ and a cubic surface $W$. Let $\tilde X = \Bl_l X$. Then $\pi : \tilde X \to \P^3$ is a conic bundle 
with a quintic degeneration locus $D$ which  has  a cubic and a quadric as irreducible components. The K3 surface $S$, which parametrizes irreducible components of degenerate conics in the fibration $\pi$, that are fixed by $\sigma$,
is a double cover of the cubic surface $W$ ramified along a degree 6 curve.There is an isomorphism of motives
$$t_2(S)(1) \simeq t(X),$$
\noindent see [BP, Rk.4.4]. \par
In [BGvBM, Prop.5.10] it  is proved that $X$ has 1120 FM-partners $Y$ such that $X$ and $Y$ are birational. Each of these cubics $Y$ has a non -symplectic involution $\tau$ which is induced
by an involution on $\P^5$ of the form 
$$\tau :  [x_0,x_1,x_2,x_3,x_4,x_5] \to  [x_0,x_1,x_2,x_3,x_4, -x_5],$$
The family $\sF'$ of cubic fourfolds which are invariant under the involution $\tau$ has dimension 14 and coincides with the family of cubic fourfolds with an Eckardt point.  If a  smooth cubic 4-fold $Y \subset \P^5$ contains an Eckardt point $p$ then we can choose coordinates $[x_0,\cdots, x_5] $ in $\P^5$ such that  $p=(0,0,0,0,0,1)$ and $Y$ is defined by an equation
$$  f(x_0,x_1,x_2,x_3,x_4) +  l(x_0,x_1,x_2,x_3,x_4) x^2_5=0.$$
\noindent where $f$ has degree 3 and  $ l(x_0,x_1,x_2,x_3,x_4)$ is a linear form. The cubic fourfold $Y$ contains a cone over the cubic surface $S =Y \cap H$, where $H$ is the hyperplane defined by $ l(x_0,x_1,x_2,x_3,x_4)=0$. The  cubic surface $S$ is isomorphic to $\P^2$ with 6 points blown up.Let $\bar F_0$ be the pull-back of a general line on $\P^2$.Denote the exceptional curves by $\bar F_1, \cdots,\bar F_6$ and let $F_0,F_1, \cdots, F_6$ be the cones over the corresponding curves on $S$ with vertex in $p$. In particular $F_1,\cdots ,F_6$ are planes in $Y$. The surfaces classes
$[F_0],[F_1],\cdots,[F_6]$ generate $A(Y) =H^4(Y,\Z) \cap H^{2,2}(Y)$ and are invariant under $\tau$. Moreover
$$h^2 =3[F_0]-[F_1]-\cdots -[F_6]. $$
 \noindent where $h$ is the class of a hyperplane section and 
$$  [F_0] \cdot [F_0] =7  \ , \  [F_0] \cdot [F_i] =3 \ , \  [F_i]\cdot F_i] =3 \ ,  \    [F_0] \cdot h^2 =3 $$ 
\noindent see [LZP, Lemma 2.4]. The rank two lattice $<h^2,[F_0]>$ has discriminant 12. Since $Y$  contains a plane we get $X \in \sC_8 \cap \sC_{12}$.\par 
According to Theorem 1.1 in [BBA] there are three irreducible components of $\sC_8 \cap \sC_{12}$ indexed by the value $P \cdot T = \epsilon \in (1,2,3)$,  where $ P \subset Y$  is a plane and $T$  the class
of a cubic  rational normal scroll (that  is   $T \cdot T =7 $ and  $T \cdot h^2 = 3$).Therefore $Y$ belongs to the irreducible component corresponding to $\epsilon =3$.  There is a K3 surface $S'$, double cover  of $\P^2$, ramified along a reduced sextic $C \cup L$, where $C$ is a quintic curve and $L$ a line, such that
$$t_2(S') (1) \simeq t(Y),$$
\noindent see [Ped, Prop.2.2]. Therefore $t_2(S) \simeq t_2(S')$ and $t(X) \simeq t(Y)$.\par
Both $X$ and $Y$ have no  Hodge theoretically associated K3 surfaces and are conjecturally irrational. \par
Next we consider the case of  FM-partners $X,Y$ that are rational and belong to $\sC_{42}$.\par
Let $\sG$  be the family of cubic fourfolds which are  invariant under a  order 3 symplectic  automorphism  
\begin{equation}  \label{order 3} \mu :  [x_0,x_1,x_2,x_3,x_4,x_5] \to [x_0,x_1,\zeta x_2,\zeta x_3,\zeta^2 x_4, \zeta^2 x_5], \end{equation}
\noindent where $\zeta$ is a cubic  root of 1. Every $X \in \sG$ has an equation of the form
$$F(x_0,x_1,x_2,x_3,x_4,x_5) = f_1(x_0,x_1) +f_2(x_2,x_3) +f_3(x_4,x_5) + \sum_{i,j,k} a_{ijk}x_i x_j x_k$$
\noindent where $f_1,f_2,f_3$  have degree 3 and $i =0,1; j= 2,3 ; k= 4,5$. The family $\sG$ has dimension 8. The lattice of algebraic cycle  $A(X)$ equals $<h^2> \oplus A(X)_{prim} $, where $A(X)_{prim}$ has rank 12.
A general $X \in \sG$ contains 378  families of cubic scrolls $\{T_i,T^*_i\}$ such that $[T_i] +[T^*_i]=2h^2$, where $h$ denotes the class of a hyperplane section.  Every $X \in \sG$ is rational and belongs to $\sC_{12} \cap \sC_{42}$, see
[BGM].\par
If $X \in \sG$ there are 623 Fourier-Mukai partners $Y$ of $X$, see  [BGvBM, 6 .Ex.2] . They all belong to $\sC_{42}$ and therefore are rational. Let
 
$$\Phi : \sF_{22} \dashrightarrow  \sC_{42}$$
\noindent  be the rational map of degree 2, with $\sF_{22}$ the moduli space of smooth polarised K3 surfaces $(S,H)$ of genus 22. Since both $X$ and its FM-partners  $Y$ belong to $\sC_{42}$ there are $S,S' \in  \sF_{22}$
such that
$$ \sA_X \simeq D^b(S)  \  ,  \  \sA_Y \simeq D^b(S')   \  ;  \  F(X) \simeq S^{[2]}  \ ,  \  F(Y) \simeq S'^{[2]}.$$
Therefore -
$$\sA_X \simeq \sA_Y \simeq D^b(S) \simeq D^b(S') .$$
The isomorphism $D^b(S) \simeq D^b(S')$ imples that $S$ and $S'$ have isomorphic Chow motives . Therefore $t_2(S) \simeq t_2(S')$ . Since  $F(X) \simeq S^{[2]}$ and $F(Y) \simeq S'^{[2]}$ there are isomorphisms of motives
$$t_2(S) (1) \simeq t(X)  \   ;   \  t_2(S')(1) \simeq  t(Y),$$
\noindent see [Ped, Lemma 5.3], which induces the isomorphism of transcendental motives  $t(X) \simeq t(Y)$.\par
The following  result gives a geometric description of the FM-partners $X$ and $Y$
\begin{prop} Let $X \in \sG$ and let $Y$ be a FM partner of $X$. The cubic fourfolds   $X$ and $Y$ belong to a   divisor $\sC_K \subset \sC_{18}$ and are rational. The cubics  $X$ and $Y$  contain  elliptic ruled surfaces $T \subset X$  and $T' \subset  Y$ such that the fibrations  in del Pezzo  surfaces of degree 6,  $\tilde X \to \P^2$ and $\tilde Y \to \P^2$, where  $\tilde X = \Bl_T(X) $ and  $\tilde Y = \Bl_{T'} (Y)$, have  a rational section.\end{prop} 
\begin{proof} In [AHTVA,Sect .4]  the Authors describe  a countable dense set of divisors $\sC_{K _{a,b}}\subset \sC_{18}$ parametrizing rational cubic fourfolds. Each cubic  in $ Y \in \sC_{K _{a,b}}$ contains a elliptic ruled surface $T$. Let
$ r : \tilde Y \to Y$ be the blow-up of $Y$ along $T$ and let $S'$ be a smooth fibre of the fibration $\tilde Y \to \P^2$. Then there is a class $\Sigma \in H^4(Y,\Z) \cap H^{2,2}(Y)$ such that  $\Sigma \cdot S =1$, with $S = r(S')$. Therefore $Y$ is rational.\par
So we are left to show that for $X \in \sG$ and a FM -partner $Y$, there are integers $a,b$ such that $X, Y \in \sC_{K _{a,b}}\subset \sC_{18}$.Here  $K_{a,b}$ is a rank 
3 positive defined lattice having Gram matrix 
$$  
\begin{pmatrix} 
3&6&a \\
6 &18&1\\
a&1&b \\
\end{pmatrix}
$$
\noindent where $a\equiv  b (2)$ and   $a=(-1,0,1)$.The lattice $K_{a,b}$ contains\par
\noindent $<h^2>^{\perp}$,where $h$ is the class of a hyperplane section  and $K_{a,b}$ embeds in the lattice $L =H^4(Y,\Z)$.
The lattices  $K_{0,b}$ are isometric to the admissible lattices   contained in $L$ with Gram matrix 

\begin{equation} \label{rank 3} A_{2,1,b-2/2} =\begin{pmatrix}
3&0&0\\
0&b &1 \\
0&1& 6 
\end{pmatrix} \end{equation}
\noindent see [YY,Rk 8.17]..\par
Therefore  if   $A(X)$  and $A(Y)$  contain an admissible rank 3 lattice with Gram matrix as in \ref{rank 3} and $b \equiv 0 (2)$, then  the fourfolds $X$ and $Y$  belong  to $\sC_{K _{0,b}}$.\par
The primitive lattices $A(X)_{prim} $ and $A(Y)_{prim}$ have rank 12  with  generators  $\alpha_1, \cdots \alpha _{12} $ and  $\beta_1,\cdots,\beta_{12} $ . Their Gram matrices  $M$ and $M'$ are described in [BGvBM, Lemma 6.3]. 
In the case of a lattice  with Gram matrix $M$  there are classes $\alpha_k, \alpha_i, \alpha _j \in A(X)_{prim}$ such that $\alpha_i \cdot h^2 = \alpha_j \cdot h^2 = \alpha_k \cdot h^2=0$, 
$\alpha^2_i = \alpha^2_j =\alpha^2_k =4 $ and 
$$ \alpha_i \cdot \alpha_j = 1  \  ,   \    \alpha_k \cdot \alpha_i = 2   \  , \   \alpha_k \cdot \alpha_j = 1 $$
Similarly if the   Gram matrix is $M'$ there are classes  $\beta_l,\beta_m,  \beta _n  \in A(Y)_{prim}$, such that
$$ h^2 \cdot \beta_l=  h^2 \cdot \beta_m= h^2 \cdot \beta_n=0  \   ,  \  \beta^2_l=\beta^2_m=\beta^2_n= 4,$$
$$\beta_l \cdot \beta_m =1  \  ,  \beta_n \cdot \beta_l = 2   \ ,  \   \beta_n \cdot  \beta_m =1$$ 
Therefore in any case the lattices 
$$K_X=<h^2,\alpha_k, (\alpha_i-\alpha_j)>   \  ;   \  K_Y =<h^2, \beta_n , (\beta_l-\beta_m)>$$ 
\noindent  are contained in $A(X)$ and $A(Y)$,respectively, and  have a Gram matrix of the form
$$\begin{pmatrix}
 3&0&0 \\
0&4&1 \\ 
0&1& 6 
\end{pmatrix} $$
\noindent which coincides with $ A_{2,1,b-2/2}$ for $b =4$. 
\end{proof}
\begin{rk} ($\L$-equivalence) Two smooth varieties over field  $k$ are said to be $\L-$equivalent in $K_0(\Var_k)$ if, for some positive integer $r$,   
$$\L^r\cdot( |X|-|Y|) =0  \   in \      K_0(\Var_k),$$
\noindent where $\L =|\A^1_k|$. In  [MM] several examples are constructed of $\L$-equivalent  special cubic fourfolds $X,Y$ that are Fourier-Mukai partners. In these examples ( see [MM 4.7 and 4.8]) the  isomorphism between the transcendental motives  $t(X) $ and $t(Y)$ is induced by an isomorphism of the transcendental  motives of the associated K3 surfaces. \end{rk}

\subsection{Fano variety of lines} In [BFM] it is conjectured that, if $X,X'$ are Fourier-Mukai partners, then the derived categories of the Fano varieties  of lines $D^b(F(X))$ and $D^b(F(X'))$ are equivalent .This is  the case   for  $X,X'$  in the divisor $\sC_{546} $, see [BFM,Thm.7.1]. In   [KS,Thm.A] the Authors prove, using matrix factorization, that for every smooth $X$ there  is an isomorphism $\Sym^2\sA_X \simeq D^b(F(X)$.Therefore the equivalence of categories 
$\sA_X \simeq \sA_{X'} $ implies $D^b(F(X) )\simeq D^b(F(X'))$.\par
Since birationally equivalent  hyper-K\"alher manifolds have isomorphic Chow motives if  for  two smooth cubic fourfolds $X,X'$ the Fano varieties $F $ and $F'$ are birationally equivalent then $h(F) =h(F') $ in $\sM_{rat}(\C)$. 
The isomorphism  $h(F) =h(F') $ implies $t(X) \simeq t(X')$, see [BP]. \par
Examples of FM-partners $X,X'$ with birationally equivalent Fano  varieties of lines are  cubic fourfolds  containing a non-syzygetic pair of cubic scrolls, i.e cubics in $\sC_{12}$ containing two cubic scrolls $T_1,T_2$ such that
$T_1 \cdot T_2=1$, see [BFM, Sect.3]. If $X$ is such a cubic fourfold in $\sC_{12}$ then there  exists a Fourier-Mukai partner $X'$, birational to $X$ in $\sC_{12}$, also containing  a non-syzygetic pair of cubic scrolls, such that $F(X)$ is birational to $F(X')$. There are associated K3 surfaces $S$ and $S'$ (in the motivic sense), such that $t(X)\simeq t_2(S)(1)$ and $t(X') \simeq t_2(S')(1)$.  Therefore the isomorphism $t(X) \simeq t(X')$ implies that $t_2(S) \simeq t_2(S')$.\par 
On the other hand there are  examples of cubic fourfolds $X,X'$ in the Hassett divisor $\sC_{32}$ such that $F(X) $ is isomorphic to $F(X'$ but $X$ and $X'$ are not FM-partners, see [BFM, Thm.6.1].\par

\section{ Equivariant Kuznetsov component}
 Let $X$ be a cubic fourfold with a group action by a finite group $G$. Then the line bundle $\sO_X(1)$ and the semiorthogonal decomposition
$$D^b(X) =<\sA_X,\sO_X,\sO_X(1),\sO_X(2)>$$
\noindent are preserved under the group action of $G$. Hence we obtain the semiorthogonal decomposition
$$D^b_G(X) =<\sA^G_X,<\sO_X>^G,<\sO_X(1)>^G,<\sO_X(2)>^G>$$
\noindent of the equivariant derived category of $X$. The equivariant Kuznetsov component $\sA^G_X$ is defined as the orthogonal complement of 
$$< \ <\sO_X>^G ,<\sO_X(1)>^G,<\sO_X(2)>^G \ >,$$
\noindent see [FFM].\par
It is natural to ask whether $\sA^G_X$ is equivalent to the derived category  of a smooth variety. Here we describe families of cubic fourfolds  $X$ with a symplectic automorphism  $\sigma$ of prime order, such that
 $$\sA^G_X \simeq D^b(S),$$
\noindent where $G = <\sigma> $ and $S$ is either a K3 surface or an abelian surface.\par 
If $X$ is a general cubic fourfold with a symplectic involution $\sigma$. Then
$$\sA^G_X \simeq D^b(S),$$
\noindent where $G =<\sigma>$ and $S \subset F(X)$ is the K3 component of the fixed locus of $G$ on the Fano variety of lines $F(X)$, see [FFM]. We also have
$$t_2(S) (1) \simeq t(X),$$
\noindent see [Ped,  Prop. 2.6].\par
Let  $\sF$ be the family of cubic fourfolds invariant under the automorphisms of $\P^5$
$$\sigma: [x_0,x_1,x_2,x_3,x_4,x_5] \to :[x_0,x_1,x_2, \zeta x_3, \zeta x_4,\zeta x_5]$$
\noindent with $\zeta^3 = 1, \zeta \ne 1$. A cubic fourfold $X \in \sF $ has a an equation of the form 
 \begin{equation} \label{eq} F(x_0,x_1,x_2,x_3, x_4,x_5) =   f(x_0,x_1,x_2) +g(x_3,x_4,x_5)=0 \end{equation}
\noindent where $f$ and $g$ are homogeneous of degree 3. Let $Z\subset \P^3$ and $T\subset \P^3$ be the cubic surfaces defined by $f(x_0,x_1,x_2) -t^3 =0$ and $ g(x_3,x_4,x_5)-t^3 =0$. The plane $t=0$ cuts a smooth cubic curve $C\subset Z$ and a smooth cubic curve $D \subset T$. By [CT,Prop.1.2]  there is a rational map $\P^3 \times \P^3 \to \P^5$ which induces a rational dominant map $\psi : Z \times T \to X$ and whose locus of indeterminacy is  the
abelian surface $C\times D$. Let $\tilde X$ be the blow-up of $Z \times T$ at $C\times D$. The fourfold $X$ contains two disjoint planes $P_1$ and $P_2$, see [CT, Rk.2.4], hence it is rational.\par
 By  Manin's formula there is an isomorphism

$$h(\tilde X)\simeq h(Z\times T) \oplus h (C \times D)(1) .$$

\noindent The motive of the fourfold $Z \times T$  has no transcendental  part, since both the surfaces $Z$ and $T$ are rational. Therefore  the transcendental part  of $h(\tilde X)$  coincides with the transcendental motive $t_2(C \times D)(1)$. The map $\psi$ induces a finite morphism $\tilde \psi: \tilde X \to X$ and hence $h(X)$ is a direct summand of $h(\tilde X)$. It follows that  
$$t(X) \simeq t_2(C \times D) (1)$$
The quotient $X /G$, with $G =<\sigma>$, has  a  crepant resolution  $Y \to X/G$.Then $D^b(Y) \simeq D^b_G(X)$ and there is an isomorphism
$$\sA^G_X \simeq D^b(C \times D)$$
\noindent see [XH,Thm.5.8]. Since $X\in \sC_8$  and is rational there is a K3 surface $S$, double cover of $\P^2$, ramified along a sextic curve $C$, such that
$$t_2(S)(1) \simeq t(X)  \  and   \   \sA_X \simeq D^b(S).$$
Therefore  $t_2 (C \times D) \simeq t_2(S)$,  while $ D^b( C\times D) \ne D^b(S)$.\par
Next we consider special cubic fourfolds  $X$ in countably many Hassett divisors, with a group $G$ of symplectic automorphisms, such that $\sA^G_X \simeq D^b(\Sigma)$, where $\Sigma$ is a K3 surface.
 \begin{prop}   Let $S $  be a K3 surface with a polarization   $L$ of degree $L^2= 6d $ and genus $g$ ,where  $6d =2g -2$ and $g =n^2 +n +2 $. Assume that $S$ has a  symplectic automorphism $\sigma$ of order 3. Let $X \in \sC_{6d} $ be the image of $S$ under the  surjective rational map
 \begin{equation} \label{polarized} \sF_{6d} \dashrightarrow \sC_{6d}, \end{equation}
\noindent where $\sF_{6d}$ is the moduli space of K3 surfaces of degree $6d$.  If  $( n^2+n+1)/3 +1= m^2 +m+2$, with $m\ge 2$, there are  a cubic fourfold $Y \in \sC_{2d} $  and a K3 surface $\Sigma$, such that
$$\sA^G_X\simeq \sA_{Y}  \  ;  \    \sA^G_X \simeq  D^b(\Sigma)  \     and        \    t(X) \simeq t(Y) $$
\noindent where $G =<\sigma>$. \end{prop}
\begin{proof} Let  $\Sigma $ be the K3 surface which is the  minimal desingularization of the quotient $S/\sigma$. The surface $S/\sigma$ has 6 singularities of type $A_2 $ and $\Sigma$ contains 12 irreducible curves, i.e. 6  disjoint pairs of rational curves meeting in a point. In the diagram
\begin{equation} \label{diagram}\CD
\tilde S@>{	\alpha}>>  S    \\
@VVV                  @V{\pi}VV  \\
\Sigma@>{\beta}>>S/\sigma    \endCD \end{equation}

\noindent the surface $\tilde S$ is the blow-up of $S$ at the six isolated fixed points $P_1,\cdots,P_6$ of $\sigma$. The automorphism $\sigma$ extends to an automorphism  of $\tilde S$ \and $\Sigma = \tilde S/\sigma$.\par
The family $\sS$ of polarized K3 surfaces with a symplectic automorphism of order 3 is the union of countably many components of dimension 7. Similarly the family $\sT$ of  K3 surfaces that are quotients of K3 surfaces with a symplectic automorphism of order 3 is the union of countably many components of dimension 7. The correspondence between K3 surfaces in $\sS$ and in $\sT$ has been described in [GM].
If  $S$ is  a K3 surface  with a polarization  $L$ of degree $L^2= 6d =2g -2 $, where $g =n^2 +n +2 $, then the desingularization $\Sigma$ of  $S/\sigma$ has a polarization $H$ of degree $2d$, see [GM,Thm.5.2].  
Therefore $H^2 = 2d = 2g' -2$, with $g'=( n^2+n+1)/3 +1$. Since  $g'=m^2 +m+2$  the   map $\phi$ in \ref{polarized} sits in a diagram
$$ \CD
\sF_{6d}@>{\phi}>>  \sC_{6d}    \\
@V{f_*}VV                  @.  \\
 \sF_{2d }@>{\phi}>>\sC_{2d}   \endCD $$
\noindent where $f_*$ is induced by the quotient map $f : S \to S/\sigma$.  The cubic fourfold $X =\phi(S)$ belongs to $\sC_{6d}$ and $F(X) \simeq S^{[2]}$. 
The order 3 automorphism $\sigma$ induces a symplectic automorphism $\sigma^{[2]}$  on $F(X)$ and   $\sA_X \simeq D^b(S)$. The image $Y =\phi(\Sigma)$ belongs to $\sC_{2d}$, $ F(Y) \simeq \Sigma^{[2]}$  
 and  $\sA_{Y} \simeq D^b(\Sigma)$.\par
Let $G= <\sigma>$ and let $D^b_G(S)$ be the derived category of $G$-equivariant coherent sheaves on $S$. Then
$ D^b(\Sigma) \simeq D^b_G(S),$
\noindent see [XH,Thm.3.1]. Therefore we get
$$  \sA^G_X \simeq \sA_{Y}  \   and     \       \sA^G_X \simeq  D^b(\Sigma) $$
 Since the transcendental motive $t_2(-)$ of a surface is a birational invariant the maps $\tilde S \to  S$   and $\tilde S \to \Sigma$    in \ref{diagram} induce a map $\theta : t_2(S )\to t_2(\Sigma)$ that is the projection onto
 a direct summand. Therefore 
$$  t_2(\Sigma)=t_2(S) \oplus N.$$
By a result of Huybrechts a symplectic automorphism acts trivially on the Chow group of 0-cycles  of a K3 surface . Since $A_*(t_2(\Sigma) =A_0(\Sigma)_{hom} $ we get $A_i(N)=0$, for all $ i$. Therefore
$N=0$,and  hence  $t_2(S) =t_2(\Sigma)$. Since $t(X) \simeq t_2(S)(1)$ and $t(Y) \simeq t_2(\Sigma)(1)$ we get $t(X) \simeq t(Y)$.
\end{proof}
\begin{ex} (1)  For  $ n=4$ we get  $d=7$ ,  $g= 22$ and $g' = 8$. Let $S $  be  a K3 surface of genus 22 and degree 42, equipped  with a symplectic automorphisms of order 3. There are two rational cubic fourfolds $X$ and $Y$, with $X \in \sC_{42}$ and $Y \in \sC_{14}$, associated to $S $ and to  the minimal resolution $\Sigma$ of $S/\sigma$, respectively, such that
$$\sA^G_X \simeq \sA_{Y}.$$
(2) For  $n=16$ we get $d=546 $ , $g =274$  and $g'=92$. Therefore if $S$ is a  polarized K3 surface of genus 274  with a symplectic automorphism of order 3  the minimal resolution $\Sigma$ of $S/\sigma$ is a polarized K3 surface of genus 92.  There are two  cubic fourfolds $X \in \sC_{546}$ and $Y \in \sC_{182}$ such that
$$\sA^G_{ X} \simeq \sA_{Y}$$
The cubic fourfold $X \in \sC_{546}$ has a unique FM-partner $X'$  and  $\sA^G_{ X'} \simeq \sA_{Y'}$, where $Y'$ is the unique FM-partner of $Y$ in $\sC_{182}$.\par
The foufolds $X,X' \in \sC_{542}$ are conjecturally birational, see [BFM,Thm. 7.1].  
\end{ex}

\end{document}